\documentclass[a4paper, 12pt, twoside]{article}

\usepackage[top=3cm, bottom=3cm, left=3cm, right=3cm]{geometry}
\usepackage[pdftex]{graphicx}
\usepackage{tikz}

\usepackage{array}
\usepackage{booktabs}

\usepackage{amsmath}
\usepackage{amsthm}
\usepackage{amssymb}

\usepackage{mathtools}
\usepackage[integrals]{wasysym}
\usepackage{stmaryrd}
\usepackage[mathcal]{euscript}
\usepackage{mathrsfs}

\usepackage[utf8]{inputenc}
\usepackage[T1]{fontenc}

\usepackage[english]{babel}

\usepackage{cite}

\usepackage{lmodern}

\usepackage{calc}
\usepackage[inline]{enumitem}
\usepackage{siunitx}
\usepackage{url}

\usepackage[linktocpage, hidelinks, bookmarks, bookmarksnumbered, pdfstartview={XYZ null null 1.00}]{hyperref}

\theoremstyle{plain}
\newtheorem{theorem}{Theorem}[section]
\newtheorem{lemma}[theorem]{Lemma}

\newtheorem{proposition}[theorem]{Proposition}
\theoremstyle{definition}
\newtheorem{definition}[theorem]{Definition}

\newtheorem{remark}[theorem]{Remark}
\newtheorem{example}[theorem]{Example}

\DeclareMathOperator{\diag}{diag}
\DeclareMathOperator{\diff}{d\!}
\DeclareMathOperator{\transp}{T}

\DeclarePairedDelimiter{\abs}{\lvert}{\rvert}
\DeclarePairedDelimiter{\floor}{\lfloor}{\rfloor}

\newcommand{\suchthat}{\ifnum\currentgrouptype=16 \mathrel{}\middle|\mathrel{}\else\mid\fi}

\numberwithin{figure}{section}
\numberwithin{equation}{section}

\begin{document}

\setlength{\parskip}{1pt plus 1pt minus 1pt}

\setlist[enumerate, 1]{label={\textnormal{(\alph*)}}, ref={(\alph*)}, leftmargin=0pt, itemindent=*}

\title{The generic multiplicity-induced-dominancy property from retarded to neutral delay-differential equations: When delay-systems characteristics meet the zeros of Kummer functions}

\author{Islam Boussaada\thanks{Universit\'e Paris-Saclay, CNRS, CentraleSup\'elec, Inria, Laboratoire des signaux et syst\`emes, 91190, Gif-sur-Yvette, France. E-mails: \{first name.last name\}@l2s.centralesupelec.fr}~\thanks{Institut Polytechnique des Sciences Avanc\'ees (IPSA), 63 boulevard de Brandebourg, 94200 Ivry-sur-Seine, France.}, Guilherme Mazanti\footnotemark[1], Silviu-Iulian Niculescu\footnotemark[1]}

\maketitle

\begin{abstract}

In this paper, which is a direct continuation and generalization of the recent works by the authors  \cite{Boussaada2020Multiplicity, MBN-2021-JDE},  we show the validity of the \emph{generic multiplicity-induced-dominancy\/} property for a general class of linear functional differential equations with a single delay, including the retarded as well as the neutral cases. The result is based on an appropriate integral representation of the corresponding characteristic quasipolynomial functions involving some appropriate degenerate hypergeometric functions. 
\end{abstract}



\hypersetup{pdftitle={}, pdfauthor={Guilherme Mazanti, Islam Boussaada, Silviu-Iulian Niculescu}, pdfkeywords={Time-delay equations, multiplicity-induced-dominancy, stability analysis, confluent hypergeometric functions, Kummer functions, spectral methods, root assignment}}

\bigskip

\noindent \textbf{Notation.} In this paper, $\mathbb N^\ast$ denotes the set of positive integers and $\mathbb N = \mathbb N^\ast \cup \{0\}$. The set of all integers is denoted by $\mathbb Z$ and, for $a, b \in \mathbb R$, we denote $\llbracket a, b\rrbracket = [a, b] \cap \mathbb Z$, with the convention that $[a, b] = \emptyset$ if $a > b$. For a complex number $s$, $\Re(s)$ and $\Im(s)$ denote its real and imaginary parts, respectively. The open left and right complex half-planes are the sets $\mathbb C_-$ and $\mathbb C_+$, respectively, defined by $\mathbb C_- = \{s \in \mathbb C \suchthat \Re(s) < 0\}$ and $\mathbb C_+ = \{s \in \mathbb C \suchthat \Re(s) > 0\}$.

Given $k, n \in \mathbb N$ with $k \leq n$, the binomial coefficient $\binom{n}{k}$ is defined as $\binom{n}{k} = \frac{n!}{k! (n-k)!}$ and this notation is extended to $k, n \in \mathbb Z$ by setting $\binom{n}{k} = 0$ when $n < 0$, $k < 0$, or $k > n$.  For $\alpha \in \mathbb C$ and $k \in \mathbb N$, $(\alpha)_k$ is the \emph{Pochhammer symbol} for the \emph{ascending factorial}, defined inductively as $(\alpha)_0 = 1$ and $(\alpha)_{k+1} = (\alpha+k) (\alpha)_k$.

Finally, for the sake of simplicity in the formulations, we consider that the indices of rows and columns of matrices start from $0$. More precisely, given $n, m \in \mathbb N^\ast$, an $n \times m$ matrix $A$ is described by its coefficients $A_{jk}$ for integers $j, k$ with $0 \leq j < n$ and $0 \leq k < m$.

\section{Introduction}

This paper addresses the asymptotic behavior of the generic delay-differential equation (DDE):
\begin{equation}
\label{MainSystTime}
y^{(n)}(t) + a_{n-1} y^{(n-1)}(t) + \dotsb + a_0 y(t) + \alpha_{m} y^{(m)}(t - \tau) + \dotsb + \alpha_0 y(t - \tau) = 0,
\end{equation}
where the unknown function $y$ is real-valued, $n$ is a positive integer, $m$ is a nonnegative integer such that $m\leq n$, $a_k, \alpha_l \in \mathbb R$ for $k \in \llbracket 0, n-1\rrbracket$ and $l \in \llbracket 0, m\rrbracket$ are constant coefficients, and $\tau > 0$ is a delay. Under the assumption that $\alpha_m\neq 0$, equation \eqref{MainSystTime} is a delay-differential equation which is said of \emph{retarded} type if $m<n$ and of \emph{neutral} type if $m=n$ (see, e.g., \cite{Hale1993Introduction} and references therein).

For linear dynamical systems including delays in their model representation, spectral methods can be used to understand the asymptotic behavior of solutions by considering the roots of some characteristic function (see, e.g., \cite{Hale1993Introduction, Michiels2014Stability, Bellman1963Differential, Cooke1986Zeroes,  Stepan1989Retarded, Wright1961Stability}) which, for \eqref{MainSystTime}, is the function $\Delta: \mathbb C \to \mathbb C$ defined for $s \in \mathbb C$ by
\begin{equation}
\label{Delta}
\Delta(s) = s^n + \sum_{k=0}^{n-1} a_k s^k + e^{-s\tau} \sum_{k=0}^{m} \alpha_k s^k.
\end{equation}
More precisely, the exponential behavior of solutions of \eqref{MainSystTime} is given by the real number $\gamma_0 = \sup\{\Re(s) \suchthat s \in \mathbb C,\; \Delta(s) = 0\}$, called the \emph{spectral abscissa} of $\Delta$, in the sense that, for every $\varepsilon > 0$, there exists $C > 0$ such that, for every solution $y$ of \eqref{MainSystTime}, one has $\abs{y(t)} \leq C e^{(\gamma_0 + \varepsilon) t} \max_{\theta \in [-\tau, 0]} \abs{y(\theta)}$ \cite[Chapter~1, Theorem~6.2 and Corollary~7.2]{Hale1993Introduction}. Moreover, all solutions of \eqref{MainSystTime} converge exponentially to $0$ if and only if $\gamma_0 < 0$. An important difficulty in the analysis of the asymptotic behavior of \eqref{MainSystTime} is that, contrarily to the delay-free case, the corresponding characteristic function $\Delta$ has infinitely many roots.

In all generality, the problem of characterizing the domain in the space of the equation's parameters that guarantee the exponential stability of solutions is a question of ongoing interest, see for instance \cite{Michiels2007Stability}. Since the 1950s, the complete understanding of the first-order retarded scalar differential equation with a single delay has benefited from the seminal work by Hayes \cite{Hayes}, where the author gives a complete characterization of the rightmost spectral value location as a function of the delay and system's parameters. More recently, several works exploited Hayes results in control problems such as in delayed feedback and in stabilization problems. Unfortunately, Hayes approach remains complicated and natural extensions to nonscalar or higher-order retarded or neutral DDEs do not exist.
 
In control theory, the first pole-placement paradigm for time-delay systems, called \emph{finite pole placement}, was introduced in the late 1970s in \cite{Olbrot, ManitiusOlbrot}, where a prediction of the state over a delay interval is set to counteract the effect of the delay, hence reducing the closed-loop system to a finite-dimensional plant. For an exhaustive presentation of the main ideas as well as some comparisons with other methods devoted to the control of dynamical systems including input delays, we refer to \cite{wang-et-al:99}. Later, such an idea was deeply investigated in a more appropriate algebraic setting by \cite{Brethe} thanks to the introduction of the ring $\mathcal{E}$ defined as the set of all the meromorphic functions in the complex plane $\mathbb{C}$ that are of the form of $P(s, e^{-h \, s})/Q(s)$, where $Q$ is a polynomial in the Laplace complex variable $s$, $P$ is a bivariate polynomial in $s$ and $e^{-h \, s}$, and $h$ is a fixed positive real number.  Indeed, the major issue for the algebraic design of controllers of linear time-invariant differential time-delay systems is the algorithmic study of the ring $\mathcal{E}$.
Furthermore, in practice, the limitation of such a paradigm was remarked during the early 2000s in \cite{Engelborghs} through the numerical design of a delayed controller theoretically able to stabilize a dynamical system described by a first-order scalar differential equation but leading to a closed-loop system whose stability is extremely sensitive to infinitesimal uncertainties. Such an instability mechanism, which is known as the \emph{spillover} problem, is explained by the infinitely many spectral values generated by the discretization of the predictor. Furthermore, these fortuitous spectral values dominate in part the  assigned spectral value. 

Pole placement for delay systems goes beyond of a quasipolynomial interpolation problem. For instance, it has been shown in the 1980s by Ackermann \cite{Ackermann1972} that $n$ poles of the system can be assigned to (some) desired positions in the complex plane by $n$ feedback parameters in the same way as in the finite-dimensional case. However, it is well known that such an interpolation is an effective placement if and only if the remaining spectral values of the closed-loop system are located to the left of the assigned poles, thus avoiding the spillover problem. In other words, such a strategy works if the assigned poles are dominant. But, as pointed out by  \cite{VMZ2010}, such a property is in general not guaranteed as commented in \cite{RAMLAA}, where the proposed pole placement relies on rather heuristic trial-and-error placement of the dominant poles, making the pole placement procedure repeated with a different selection of assigned poles without any attempt to prove analytically the corresponding property.

To the best of our knowledge, the first ``automated'' pole placement for retarded time-delay systems is the numerical paradigm known as \emph{continuous pole placement} introduced in \cite{Michiels2002}. Unlike finite spectrum assignment, continuous pole placement does not render the closed-loop system finite-dimensional, but consists instead of controlling the corresponding rightmost eigenvalues. Such an idea represents a simple generalization of the pole placement for finite-dimensional systems represented by ordinary differential equations. It is based on the continuous dependence of the characteristic roots on the controller parameters, and the control strategy can be summarized as follows: ``Shift'' the unstable characteristic roots from $\mathbb{C}_+$ to $\mathbb{C}_-$ in a ``quasi-continuous'' way subject to the strong constraint that, during this shifting action, stable characteristic roots are not crossing the imaginary axis from $\mathbb{C}_-$ to $\mathbb{C}_+$. We refer to \cite{Michiels2014Stability} and references therein for further insights on the number of controlled characteristic roots (which is related to the available degrees of freedom induced by the controller structure) as well as the interpretation of continuous pole placement as a local strategy to solve an appropriate optimization problem where the objective function (rightmost root) is not differentiable. It is worth mentioning that continuous pole placement, initially applied to delay systems of retarded type, was extended to neutral systems in \cite{MV2005}.

A more recent pole placement analytical paradigm ensues from an observation on the effect of multiple spectral values on the stability of DDEs, a property called \emph{multiplicity-induced-dominancy} (MID) in \cite{Boussaada2020Multiplicity}. Indeed, some works have shown that, for some classes of time-delay systems, a real root of maximal multiplicity is necessarily the rightmost root, a property we call \emph{generic multiplicity-induced-dominancy}, or GMID for short. This link between maximal multiplicity and dominance has been suggested in \cite{Pinney1958Ordinary} after the study of some simple, low-order cases, but without any attempt to address the general case. To the best of the authors' knowledge, very few works have considered this question in more details until recently in works such as \cite{Boussaada2016Multiplicity,Boussaada2020Multiplicity,Boussaada2018Further, Ramirez2016Design,MBN-2021-JDE, MBNC-2021-MTNS, Benarab2020MID, mazantiComplex}. These works consider only DDEs with a single delay and show either the MID or the GMID property for each system under consideration. For instance, the MID or GMID properties are shown to hold for retarded equations of order $1$ in \cite{Boussaada2016Multiplicity}, which proves dominance by introducing a factorization of $\Delta$ in terms of an integral expression when it admits a root of maximal multiplicity $2$; for retarded equations of order $2$ with a delayed term of order zero in \cite{Boussaada2018Further}, using also the same factorization technique; or for retarded equations of order $2$ with a delayed term of order $1$ in \cite{Boussaada2020Multiplicity}, where both the MID and the GMID properties are investigated, using Cauchy's argument principle to prove dominance of the multiple root. Most of these results are actually particular cases of a more general result on the GMID property from \cite{MBN-2021-JDE} for generic retarded DDEs of order $n$ with delayed ``term'' (polynomial) of order $n-1$ (i.e. $m=n-1$ in \eqref{MainSystTime}), which relies on links between quasipolynomials with a real root of maximal multiplicity and the Kummer confluent hypergeometric function in terms of the location of the characteristic roots. The GMID property was also extended to neutral DDEs of orders $1$ and $2$ in \cite{Ma,Benarab2020MID, MBNC-2021-MTNS}, as well as to the case of complex conjugate roots of maximal multiplicity in \cite{mazantiComplex}.

The fact that a spectral value achieves maximum multiplicity imposes algebraic constraints on each of the system's ``entries'' (polynomial coefficients as well as the delay parameter). An MID-based approach is proposed in \cite{Balogh20,Balogh21} operating the intimate representation of the quasipolynomial to provide conditions for one spectral value with an eligible intermediate multiplicity. This makes it possible to split the system parameters into two categories, some of them considered as model parameters (assumed to be fixed and known) and the remaining ones considered as values to be adjusted. Such a classification opens interesting perspectives in control design, such as the systematic tuning of the gains of the well-known Proportional-Integral-Derivative (PID) controller, able to stabilize single-input/single-output plants including one delay in the input/output channel, as suggested in \cite{Ma}.

The contributions of this paper are fourfold. First, we propose an \emph{unified analytical framework\/} for the characterization of the \emph{generic multiplicity-induced-dominancy\/} of dynamical systems represented by DDEs in both retarded and neutral cases. The proposed methodology makes use of some \emph{degenerate hypergeometric functions} and the existing links between such functions and the factorization of the characteristic functions of DDEs. In this sense, a first step was proposed by the same authors in \cite{MBN-2021-JDE}, where only the special case of retarded equations with $m = n - 1$ was treated by using a result by Wynn \cite{Wynn1973Zeros} on the location of zeros of some Kummer confluent hypergeometric functions, the latter being proved using a representation of a quotient of Kummer functions as continued fractions. However, the underlying methodology from \cite{Wynn1973Zeros} seems hard to extend to more general retarded cases or to the neutral one, and for this reason the present paper is based on a different approach.

Second, our method is constructive and allows a better understanding of the existing links between the location of the characteristic root with generic maximal multiplicity, the so-called rightmost root, i.e., the spectral abscissa of the corresponding system, and the parameters of the dynamical systems. In particular, in some cases, it may be useful to minimize the spectral abscissa with respect to some of the tuning parameters with a guarantee of the exponential stability of the zero solution if the characteristic root satisfying the GMID property is located in $\mathbb{C}_-$.

Third, in terms of control, as mentioned above, the paradigm of \emph{generic multiplicity-induced-dominancy\/} opens the perspective to a new control method, the so-called \emph{partial pole placement}. More precisely, the GMID idea may be particularly adapted for tuning low-complexity controllers, i.e., controllers including a small number of parameters (including also the delay among the parameters) with a guarantee on the location of the remaining characteristic roots for the closed-loop system. For illustrating such an idea, some examples of PID controllers including a delay in the input/output channel are proposed for low-order systems in both retarded and neutral cases.   

Finally, the ideas above can be tested and illustrated by using the software ``Partial pole placement via delay action'', or P3$\delta$\footnote{\url{https://cutt.ly/p3delta}} for short, which offers the computation of the values of the free parameters (typically tuning the controller gains) ensuring a prescribed multiplicity, establishes a certified assignment region for the rightmost root, performs a numerical computation to illustrate the distribution of the quasipolynomial characteristics, offers the possibility of a numerical study of the sensitivity of the spectrum with respect to uncertain parameters variations, and is also able to perform appropriate time-domain simulations, see  for instance  \cite{BoussaadaPartial}. New functionalities are now available covering both retarded and neutral cases.

The remaining of the paper is organized as follows. Section~\ref{SecPrelim} presents some prerequisites in complex analysis. It starts by recalling some qualitative properties of quasipolynomials  and states some technical results on the distribution of zeros of some degenerate hypergeometric functions. The main results are presented in Section~\ref{SecMainResults}, where the GMID property is proved for the general retarded and neutral differential equations with a single delay. These results, which are implemented in the control design software  P3$\delta$, helped us to study standard examples in the design of stabilizing proportional-integral-derivative controllers as in \cite{Ma}. Section~\ref{SecIllustrative} provides some illustrations of the ideas above applied to a couple of comprehensive control problems. More precisely, the stabilization of the classical pendulum and the feedback stabilization for a scalar conservation law with PI boundary control, a description of the new features of the P3$\delta$ software as well as some  further remarks on the optimization of the trivial solution decay are presented. Some concluding remarks end the paper.

\section{Preliminaries and prerequisites}
\label{SecPrelim}

This section recalls some preliminary results on properties of quasipolynomial functions (Section~\ref{SecQuasipoly}), the distribution of zeros of Laplace transform of Bernstein polynomials (Section~\ref{LaplaceZero}), and properties and the distribution of nonasymptotic zeros of degenerate hypergeometric functions (Section~\ref{SecHypergeom}).

\subsection{Quasipolynomials}
\label{SecQuasipoly}

Let us start by recalling the classical definition of 
a quasipolynomial, also known as exponential polynomial,
and its degree (see, e.g., \cite{Wielonsky2001Rolle, Berenstein1995Complex}).

A \emph{quasipolynomial} $Q$ is an entire function $Q: \mathbb C \to \mathbb C$ which can be written under the form
\begin{equation}
\label{GenericQuasipolynomial}
Q(s) = \sum_{k = 0}^{\ell} P_k(s) e^{\sigma_k s},
\end{equation}
where $\ell$ is a nonnegative integer, $\sigma_0, \dotsc, \sigma_\ell$ are pairwise distinct real numbers, and, for $k \in \llbracket 0, \ell\rrbracket$, $P_k$ is a nonzero polynomial of degree $d_k \geq 0$ with coefficients belonging to $\mathbb{C}$. The integer $\mathscr{D}_{PS} = \ell + \sum_{k=0}^\ell d_k$ 
is called the \emph{degree} of $Q$.
 When $\sigma_0 = 0$ and $\sigma_k < 0$ for $k \in \llbracket 1, \ell\rrbracket$ in the above definition, $Q$ is the characteristic function of a linear time-delay system with delays $-\sigma_1, \dotsc, -\sigma_\ell$, which motivates our study of such quasipolynomials.

Notice that, if the coefficients of a quasipolynomial $Q$ are real, which is the case in \eqref{Delta}, then the corresponding zeros are necessarily either real or complex-conjugate pairs.
The roots of a quasipolynomial do not change when its coefficients are all multiplied by the same nonzero number, and hence one may always assume, without loss of generality, that one nonzero coefficient of a quasipolynomial is normalized to $1$, such as the coefficient of the term of highest degree in $P_0$, which is the case in \eqref{Delta}, for instance. The degree of a quasipolynomial is the number of the remaining coefficients. Since each polynomial $P_k$ of degree $d_k$ has $d_k+1$ coefficients, the quasipolynomial degree is then the sum of these numbers discounting the normalized coefficient, giving rise to the number $\mathscr{D}_{PS} = \sum_{k=0}^\ell (d_k + 1) - 1 = \ell + \sum_{k = 0}^\ell d_k$.

It is worth to note that a quasipolynomial admits an infinite number of roots, except in trivial cases when the quasipolynomial reduces to a polynomial. Fortunately, there does exist a link between the degree of a quasipolynomial and the number of its roots in horizontal strips of the complex plane, thanks to a classical result known as \emph{P\'{o}lya--Szeg\H{o} bound} (see \cite[Part Three, Problem~206.2]{Polya1998Problems}), which we state in the next proposition.

\begin{proposition}
\label{PropPolyaSzego}
Let $Q$ be a quasipolynomial of degree $\mathscr{D}_{PS}$ given under the form \eqref{GenericQuasipolynomial}, $\alpha, \beta \in \mathbb R$ be such that $\alpha \leq \beta$, and $\sigma_\delta = \max_{j, k \in \llbracket 0, \ell\rrbracket} \sigma_j - \sigma_k$. Let $m_{\alpha, \beta}$ denote the number of roots of $Q$ contained in the set $\{s \in \mathbb C \suchthat \alpha \leq \Im(s) \leq \beta\}$ counting multiplicities. Then
\[
\frac{\sigma_\delta (\beta - \alpha)}{2 \pi} - \mathscr{D}_{PS} \leq m_{\alpha, \beta} \leq \frac{\sigma_\delta (\beta - \alpha)}{2 \pi} + \mathscr{D}_{PS}.
\]
\end{proposition}

Notice that the above result was first introduced and claimed in the problems collection published in 1925 by G.~P\'{o}lya and G.~Szeg\H{o}. In the fourth edition of their book \cite[Part Three, Problem~206.2]{Polya1998Problems}, G.~P\'{o}lya and G.~Szeg\H{o} emphasized that the proof was obtained in the meantime by N.~Obreschkoff using the principle argument, see \cite{Obreschkoff}.

As an immediate consequence of Proposition~\ref{PropPolyaSzego}, given a root $s_0 \in \mathbb C$ of a quasipolynomial $Q$ of degree $\mathscr{D}_{PS}$, by letting $\beta = \alpha = \Im(s_0)$ in the statement of Proposition~\ref{PropPolyaSzego}, one concludes that $s_0$ has multiplicity at most $\mathscr{D}_{PS}$. Hence, for a quasipolynomial $\Delta$ under the form \eqref{Delta}, which is of degree $m+n+1$, any of its roots has multiplicity at most $m+n+1$. 

Since quasipolynomial functions admits an infinite number of zeros, an important information on the distribution of such zeros is the location of the corresponding dominant or rightmost root. 
Let us consider a function $Q: \mathbb C \to \mathbb C$ and a value $s_0 \in \mathbb C$ such that $Q(s_0) = 0$. We say that $s_0$ is a \emph{dominant} (respectively, \emph{strictly dominant}) root of $Q$ if, for every $s \in \mathbb C \setminus \{s_0\}$ such that $Q(s) = 0$, one has $\Re(s) \leq \Re(s_0)$ (respectively, $\Re(s) < \Re(s_0)$).

The next lemma, which is a key ingredient to simplify the proofs of our main results in Section~\ref{SecMainResults}, describes how the coefficients of a quasipolynomial change under a translation and a dilation of the complex plane. Its proof, omitted here, can be carried out by straightforward computations and is very similar to that of \cite[Lemma~4.1]{MBN-2021-JDE}.

\begin{lemma}
\label{LemmDeltaTilde}
Let $s_0 \in \mathbb R$, $\Delta$ be the quasipolynomial from \eqref{Delta}, and consider the quasipolynomial $\widetilde\Delta: \mathbb C \to \mathbb C$ obtained from $\Delta$ by the change of variables $z = \tau(s - s_0)$ and multiplication by $\tau^n$, i.e.,
\begin{equation}
\label{DefiDeltaTilde}
\widetilde\Delta(z) = \tau^n \Delta\left(s_0 + \tfrac{z}{\tau}\right).
\end{equation}
Then
\begin{equation}
\label{DeltaTilde}
\widetilde\Delta(z) = z^n + \sum_{k=0}^{n-1} b_k z^k + e^{-z} \sum_{k=0}^{m} \beta_k z^k,
\end{equation}
where
\begin{equation}
\label{RelationBA}
\left\{
\begin{aligned}
b_k & = \binom{n}{k} \tau^{n-k} s_0^{n-k} + \tau^{n-k} \sum_{j=k}^{n-1} \binom{j}{k} s_0^{j-k} a_j, & \quad & \textnormal{for every } k \in \llbracket 0, n-1\rrbracket, \\
\beta_k & =  \tau^{n - k} e^{- s_0 \tau} \sum_{j=k}^{m} \binom{j}{k} s_0^{j - k} \alpha_j, & & \textnormal{for every } k \in \llbracket 0, m\rrbracket.
\end{aligned}
\right.
\end{equation}
\end{lemma}

The relations between the coefficients $b_0, \dotsc, b_{n-1}, \beta_0, \dotsc, \beta_{m}$ and $a_0,\allowbreak \dotsc,\allowbreak a_{n-1},\allowbreak \alpha_0,\allowbreak \dotsc,\allowbreak \alpha_{m}$ expressed in \eqref{RelationBA} can be expressed under matrix form as
\[
b = T_{n} a + v, \qquad \beta = \tau^{n - m - 1} e^{-s_0 \tau} T_{m+1} \alpha,
\]
where
\[
b = \begin{pmatrix}
b_0 \\
\vdots \\
b_{n-1} \\
\end{pmatrix}, \quad \beta = \begin{pmatrix}
\beta_0 \\
\vdots \\
\beta_{m} \\
\end{pmatrix}, \quad a = \begin{pmatrix}
a_0 \\
\vdots \\
a_{n-1} \\
\end{pmatrix}, \quad \alpha = \begin{pmatrix}
\alpha_0 \\
\vdots \\
\alpha_{m} \\
\end{pmatrix}, \quad v = \begin{pmatrix}
\binom{n}{0} \tau^n s_0^n \\
\binom{n}{1} \tau^{n-1} s_0^{n-1} \\
\vdots \\
\binom{n}{n-1} \tau s_0
\end{pmatrix},
\]
and, for $ k \in \mathbb N^\ast$, the matrix $T_k \in \mathcal M_k(\mathbb R)$ is defined by
\begin{equation*}
T_k = \begin{pmatrix}
\binom{0}{0} \tau^k & \binom{1}{0} \tau^k s_0 & \binom{2}{0} \tau^k s_0^2 & \cdots & \binom{k-2}{0} \tau^k s_0^{k-2} & \binom{k-1}{0} \tau^k s_0^{k-1} \\
0 & \binom{1}{1} \tau^{k-1} & \binom{2}{1} \tau^{k-1} s_0 & \cdots & \binom{k-2}{1} \tau^{k-1} s_0^{k-3} & \binom{k-1}{1} \tau^{k-1} s_0^{k-2} \\
0 & 0 & \binom{2}{2} \tau^{k-2} & \cdots & \binom{k-2}{2} \tau^{k-2} s_0^{k-4} & \binom{k-1}{2} \tau^{k-2} s_0^{k-3} \\
\vdots & \vdots & \vdots & \ddots & \vdots & \vdots \\
0 & 0 & 0 & \cdots & \binom{k-2}{k-2} \tau^{2} & \binom{k-1}{k-2} \tau^{2} s_0 \\
0 & 0 & 0 & \cdots & 0 & \binom{k-1}{k-1} \tau \\
\end{pmatrix}.
\end{equation*}
Noticing that the confluent functional Vandermonde matrix $T_k$ is invertible for every $k \in \mathbb N^\ast$, we may thus express $a$ and $\alpha$ in terms of $b$ and $\beta$ as
\begin{equation}
\label{RelationABMatrix}
a = T_{n}^{-1}(b - v), \qquad \alpha = \tau^{m + 1 - n} e^{s_0 \tau} T_{m+1}^{-1} \beta.
\end{equation}
As computed in \cite[Lemma~4.2]{MBN-2021-JDE}, the inverse of $T_k$ is the matrix $T_k^{-1} = (S^{(k)}_{j, \ell})_{j, \ell \in \llbracket 0, k-1\rrbracket}$ whose coefficients are given, for $j, \ell \in \llbracket 0, k-1\rrbracket$, by
\[
S^{(k)}_{j, \ell} = \begin{dcases*}
0, & if $j > \ell$, \\
(-1)^{\ell-j} \binom{\ell}{j} \frac{1}{\tau^{k-\ell}} s_0^{\ell - j}, & if $j \leq \ell$.
\end{dcases*}
\]
As a consequence, we can provide explicit expressions for $a_0, \dotsc, a_{n-1}, \alpha_0, \dotsc, \alpha_m$ in terms of $b_0, \dotsc, b_{n-1}, \beta_0, \dotsc, \beta_m$, which are given in the next lemma. Its proof, omitted here, follows the same lines as that of \cite[Lemma~4.2]{MBN-2021-JDE} and is based on explicit computations from \eqref{RelationABMatrix} and the identity $\sum_{j=k}^{n-1} (-1)^{j - k} \binom{j}{k} \binom{n}{j} = - (-1)^{n - k} \binom{n}{k}$, which can be found, e.g., in \cite[Proposition~2.11]{MBN-2021-JDE}.

\begin{lemma}
\label{LemmCoeffsInverseTransform}
Let $\tau > 0$, $s_0 \in \mathbb R$, and $a_0, \dotsc, a_{n-1}, \alpha_0, \dotsc, \alpha_{n-1}, b_0, \dotsc, b_{m}, \beta_0, \dotsc, \beta_{m}$ be real numbers satisfying \eqref{RelationBA}. Then, the following equalities hold:
\begin{equation*}
\left\{
\begin{aligned}
a_k & = \binom{n}{k} (-s_0)^{n - k} + \sum_{j=k}^{n-1} (-1)^{j-k} \binom{j}{k} \frac{s_0^{j-k}}{\tau^{n-j}} b_j, & \quad & \textnormal{for every } k \in \llbracket 0, n-1\rrbracket, \\
\alpha_k & = e^{s_0 \tau} \sum_{j=k}^{m} (-1)^{j-k} \binom{j}{k} \frac{s_0^{j-k}}{\tau^{n-j}} \beta_j, & & \textnormal{for every } k \in \llbracket 0, m\rrbracket.
\end{aligned}
\right.
\end{equation*}
\end{lemma}

\subsection{Zeros' distribution of finite Laplace transform}\label{LaplaceZero}

We discuss in this section Laplace transforms of some integrable functions $f$ concentrated on a finite support, which, without loss of generality, can be taken as the interval $(0, 1)$, i.e.,
\begin{equation}\label{LAP}
    F(s)=\int_0^1 e^{s\,t}\, f(t) \diff t.
\end{equation}

Before turning to the core of this section, we present the following result on the integral of the product of a polynomial and an exponential, which is rather simple but of crucial importance to prove our main result. Its proof is straightforward and can also be found in \cite[Proposition~2.1]{MBN-2021-JDE}.

\begin{proposition}
\label{PropIntegralPExp}
Let $d \in \mathbb N$ and $p$ be a polynomial of degree at most $d$. Then, for every $z \in \mathbb C \setminus \{0\}$,
\begin{equation*}
\int_0^1 p(t) e^{-zt} \diff t = \sum_{k=0}^d \frac{p^{(k)}(0) - p^{(k)}(1) e^{-z}}{z^{k+1}}.
\end{equation*}
\end{proposition}

The question of locating the zeros of functions $F$ under the form \eqref{LAP} is not new and was for instance the subject of several pioneering works by Hardy \cite{hardy1905}, P\'{o}lya \cite{polya1918}, and Titchmarsh \cite{titchmarsh1926}. Furthermore, beyond its strict mathematical importance, the location of zeros of \eqref{LAP} is related to a wide range of problems coming from physics and engineering. Despite having been established more than a century ago, the following result from \cite{polya1918} remains relevant and can be found in \cite[Part Five, Chapter~3, Problem~177, page~66]{PSII} (see also \cite{sedletskii2004} for further related results).

\begin{theorem}[G.~P\'{o}lya, 1918]\label{Polya}
Let $f$ be a positive and continuously differentiable function defined in the interval $[0, 1]$ and satisfying $f^\prime(t) < 0$ for every $t \in [0, 1]$. Consider the function $F: \mathbb C \to \mathbb C$ defined from $f$ as in \eqref{LAP}. Then all the zeros of $F$ lie in the open right half-plane $\mathbb C_+$.
\end{theorem}

The assumptions on $f$ in Theorem~\ref{Polya} (namely, positivity and monotonicity of the function $f$) are quite restrictive for our purpose. Similar results in locating the zeros of \eqref{LAP} can be obtained when $f$ is a \emph{Bernstein polynomial} \cite{bernstein1912}, that is,
\begin{equation}\label{BernStein}
 f(t)=\binom{m}{n}\,t^m\,(1-t)^n \qquad\text{where}\quad m,n\in\mathbb{N},
 \end{equation} see also \cite{BARRIO2004} for further insights on Bernstein polynomials. The following simple example provides an illustration of that fact.

\begin{example}[Zeros' frequency bound approach]\label{LaplaceBernstein}
Consider the function $F: \mathbb C \to \mathbb C$ defined for $s \in \mathbb C$ by
\begin{equation}\label{LAPB}
    F(s)=\int_0^1 t\,(1-t)^2\,e^{s\,t} \diff t.
\end{equation}
We will show that all zeros of $F$ lie in the open right half-plane $\mathbb C_+$.

Notice first that, using Proposition \ref{PropIntegralPExp}, one obtains
\begin{equation}\label{Ipp}
   F(s)= {\frac { \left( 2\,s-6 \right) {{e}^{s}}+{s}^{2}+4\,s+6}{{s}^{4}}}.
\end{equation}
Obviously, $F$ is an analytic function since the numerator of the right-hand side of \eqref{Ipp} admits a root at zero with multiplicity four. Noticing that this numerator is a quasipolynomial of degree $4$, we deduce, as a consequence of Proposition~\ref{PropPolyaSzego}, that the function $F$ does not admit any real root. One can also show that no zero of  $F$ lies on the imaginary axis. As a matter of fact,  assume that there exists a root $s_0$ of $F$ such that  $s_0=i\,\omega_0$ with $\omega_0\neq 0$. Then
\begin{equation*}
s^4\,F(i\,\omega_0)=-6\,\cos \left( \omega_0 \right) -2\,\omega_0\,\sin \left( \omega_0 \right) 
-{\omega_0}^{2}+6+i \left( 2\,\omega_0\,\cos \left( \omega_0 \right) -6\,
\sin \left( \omega_0 \right) +4\,\omega_0 \right)=0.
\end{equation*}
 The vanishing of the corresponding real and imaginary parts  allows to eliminate the trigonometric functions as rational expressions of the crossing frequency $\omega_0$.
 Namely, one obtains
 \begin{equation*}
 \cos \left( \omega_0 \right) ={\frac {-7\,\omega_0^{2}+18}{2\,
\omega_0^{2}+18}},\;\sin \left( \omega_0 \right) =-{\frac {\omega_0\, \left( 
\omega_0^{2}-18 \right) }{2\,\omega_0^{2}+18}}.
\end{equation*}
Further, using the standard property $\sin^2(\omega_0)+\cos^2(\omega_0)=1$, one deduces that $\omega_0$ must satisfy $\omega_0^4 (\omega_0^2 + 9) = 0$, which is impossible since $\omega_0 \neq 0$, yielding the required conclusion.

We are left to show that $F$ has no zeros with negative real part. To investigate the potential roots of $F$ with negative real parts we rather investigate the roots of $s \mapsto F(-s)$ with positive real parts. Indeed, define the auxiliary function $G: \mathbb C \to \mathbb C$ by
\begin{equation}\label{OIPP}
    G(s)=e^s\,{s}^{4}\,F(-s)=\left( {s}^{2}-4\,s+6 \right) {{e}^{s}}-2\,s-6.
\end{equation}
Obviously, any nontrivial zero $s_0=\xi+i\,\omega$ of $G$ satisfies $F(-s_0)=0$. We will first show that any such $s_0$ with positive real part must necessarily satisfy $0 < \abs{\omega} < \pi$.

Since $\xi$ is assumed to be positive, we have $e^{2 \xi} > 1 + 2 \xi$ and then one deduces from equation \eqref{OIPP} that the pair $(\xi,\omega)$  necessarily satisfies
\begin{equation*}
\abs*{s_0^{2}-4\,s_0+6}^{2}\left(1+2\xi\right)-4\,\abs*{s_0+3}^{2}<0
\end{equation*}
which, using the notation $\Omega=\omega^2$, is equivalent to $H_\xi(\Omega)<0$, where
\begin{equation*}
  H_\xi(\Omega)=\left( 1+2\,\xi \right) {\Omega}^{2}+ 2\left( 2\,{\xi}-7 \right)\,{\xi}^{
2}\, \Omega+\left(2\,{\xi}^{3}-15\,{\xi}^{2}+48\,{\xi}-72\right)\,{\xi}^{
2}.
\end{equation*}
Since the leading coefficient of the polynomial $H_\xi$ is positive, then the inequality $H_\xi(\Omega)<0$ holds if and only if $H_\xi$ admits two distinct real roots and $\Omega$ belongs to the open interval whose extremities are those two roots. The polynomial $H_\xi$ admits two distinct real roots if and only if $\xi\in\left(0,\,3\left(1+\sqrt{2}\right)/2\right)$, and in that case those roots are $\Omega_{\pm}={ {\xi} \left(7\,\xi -2\,{\xi}^{2}\pm 2\,\sqrt {-8\,{\xi}^{2}+24\,\xi+18} \right)/\left({1+2\,\xi}\right) }$. Notice that $\Omega_-$ is necessarily negative and, since $\Omega > 0$, only $\Omega_+$ is of interest. Using the fact that $\sqrt {
-8\,{\xi}^{2}+24\,\xi+18} \leq 6$ one easily shows that, for every $\xi\in\left(0,\,3\left(1+\sqrt{2}\right)/2\right)$, we have the inequality $\Omega_+\leq {{ \left( -2\,{\xi}^{2}+7\,\xi+12 \right) \xi}/\left({1+2\,\xi}\right)}<\pi^2$, which finally implies that $\abs{\omega} < \pi$, as required.

We can now use the fact that any root $s_0 = \xi + i \omega$ of $G$ with $\xi > 0$ satisfies $0 < \abs{\omega} < \pi$ to deduce that no such root can exist. Indeed, taking the imaginary part of the equality $F(-s_0) = 0$ and using \eqref{LAPB}, we get
\begin{equation*}
    0=\int_{0}^{1}{{e}^{-t\xi}}\sin \left(\omega t \right) t \left( 1-t
 \right) ^{2}\diff t.
\end{equation*}
The fact that $\abs{\omega} < \pi$ implies that the integrand is positive, yielding the result. Hence all roots $s$ of \eqref{Ipp} satisfy $\Re(s)>0$.
\end{example}

The location of roots of functions of type \eqref{BernStein} with indices $m$ and $n$ which are not necessarily integers remains possible thanks to the existing link between Bernstein polynomials with \emph{degenerate hypergeometric functions}, which is the topic of our next session.

\subsection{Degenerate hypergeometric functions}
\label{SecHypergeom}

As it will be proved in Section~\ref{SecMaxMultAndFact}, when the quasipolynomial $\Delta$ from \eqref{Delta} admits a root of maximal multiplicity $m+n+1$, it can be represented in terms of a \emph{Kummer confluent hypergeometric function}, whose nontrivial roots coincide with the nontrivial roots of another special function, known as \emph{Whittaker function}. These families of special functions have been extensively studied in the literature, with, in particular, a wide range of results on the properties of the asymptotic distribution of their zeros (see, e.g., \cite{Buchholz1969Confluent}, \cite[Chapter~VI]{Erdelyi1981Higher}, \cite[Chapter~13]{Olver2010NIST}). For the purposes of our paper, however, we need information on the location of all zeros of such functions, and not only their asymptotic distribution, a subject that has been considered in fewer works in the literature (we refer the interested reader to the introduction of \cite{BMN-2021-NM} for a more detailed description on these questions). This section provides a brief presentation of the results that shall be of use in the sequel. We start by recalling the definition of Kummer confluent hypergeometric functions used in this paper.

\begin{definition}
Let $a, b \in \mathbb C$ and assume that $b$ is not a nonpositive integer. The {Kummer confluent hypergeometric function} $\Phi(a, b, \cdot): \mathbb C \to \mathbb C$ is the entire function defined for $z \in \mathbb C$ by the series
\begin{equation}
\label{DefiConfluent}
\Phi(a, b, z) = \sum_{k=0}^{\infty} \frac{(a)_k}{(b)_k} \frac{z^k}{k!},
\end{equation}
where we recall that, for $\alpha \in \mathbb C$ and $k \in \mathbb N$, $(\alpha)_k$ is the Pochhammer symbol (defined in the ``Notation'' section).
\end{definition}

\begin{remark}
Note that the series in \eqref{DefiConfluent} converges for every $z \in \mathbb C$. As presented in \cite{Buchholz1969Confluent, Erdelyi1981Higher, Olver2010NIST}, the function $\Phi(a, b, \cdot)$ satisfies the \emph{Kummer differential equation}
\begin{equation}
\label{KummerODE}
z \frac{\partial^2 \Phi}{\partial z^2}(a, b, z) + (b - z) \frac{\partial \Phi}{\partial z}(a, b, z) - a \Phi(a, b, z) = 0,
\end{equation}
which has a regular singular point at $z = 0$ and an irregular singular point at $z =\infty$.
It is well known that \eqref{KummerODE} admits two  linearly independent solutions, both of them usually called Kummer confluent hypergeometric functions. For our purpose we are concerned only with the solution \eqref{DefiConfluent}.
\end{remark}

We shall need the following classical integral representation of $\Phi$, which can be found, for instance, in \cite{Buchholz1969Confluent, Erdelyi1981Higher, Olver2010NIST, DRISSI2019}.

\begin{proposition}
Let $a, b \in \mathbb C$ and assume that $\Re(b) > \Re(a) > 0$. Then, for every $z \in \mathbb C$, we have
\begin{equation}
\label{EqKummerIntegral}
\Phi(a, b, z) = \frac{\Gamma(b)}{\Gamma(a) \Gamma(b - a)} \int_0^1 e^{zt} t^{a-1} (1-t)^{b-a-1} \diff t,
\end{equation}
where $\Gamma$ denotes the Gamma function.
\end{proposition}

We also recall (see, e.g., \cite[(13.2.39)]{Olver2010NIST}), that, for every $a, b, z \in \mathbb C$ such that $b$ is not a nonpositive integer, we have
\begin{equation*}
\Phi(a, b, z) = e^z \Phi(b-a, b, -z).
\end{equation*}
The following result, which is proved in \cite{BMN-2021-NM} using the Green--Hille transformation from \cite{hille1922}, gives insights on the distribution of the nonasymptotic zeros of Kummer hypergeometric functions with real arguments $a$ and $b$. 

\begin{proposition}
\label{CorZerosKummer}
Let $a,\,b\in \mathbb R$ be such that $b\geq 2$.
\begin{enumerate}
\item\label{CorZerosKummer-k-eq-0} If $b = 2a$, then all nontrivial roots $z$ of $\Phi(a,b, \cdot)$ are purely imaginary. 
\item\label{CorZerosKummer-k-geq-0} If $b > 2a$, then all nontrivial roots $z$ of $\Phi(a,b, \cdot)$ satisfy $\Re(z) > 0$.
\item\label{CorZerosKummer-k-leq-0} If $b < 2a$, then all nontrivial roots $z$ of $\Phi(a,b, \cdot)$ satisfy $\Re(z) < 0$.
\item\label{CorZerosKummer-k-neq-0} If $b \neq 2a$, then all nontrivial roots $z$ of $\Phi(a,b, \cdot)$ satisfy
\[(b - 2a)^2 {\Im(z)}^2 - \left(4a(b-a) - 2b\right) {\Re(z)}^{2} > 0.\]
\end{enumerate}
\end{proposition}

\begin{remark}
In the case $a \in \{\alpha, \alpha + 1\}$ and $b = 2\alpha + 1$ for some $\alpha > -\frac{1}{2}$, the conclusions of Proposition~\ref{CorZerosKummer}\ref{CorZerosKummer-k-geq-0} and \ref{CorZerosKummer-k-leq-0} were shown by P.~Wynn in \cite[Theorem~1]{Wynn1973Zeros}. The techniques used in that reference do not rely on Hille's approach, but use instead a continued fraction representation of a ratio of Kummer functions (see \cite[Theorem~B(ii)]{Wynn1973Zeros}). Note that Proposition~\ref{CorZerosKummer} does not cover all cases of Wynn's result, since the assumption $b \geq 2$ is equivalent to $\alpha \geq \frac{1}{2}$, whereas Wynn's result is obtained for all $\alpha > -\frac{1}{2}$.
\end{remark}

\section{Main results}
\label{SecMainResults}

This section is dedicated to our main results on the location of zeros of the characteristic function $\Delta$ given in \eqref{Delta} of system \eqref{MainSystTime}.
Hereafter, we will characterize the manifold corresponding to the existence of a spectral value of \eqref{Delta} with multiplicity reaching the P\'{o}lya--Szeg\H{o} bound $\mathscr{D}_{PS}=m+n+1$ and will prove that such spectral value is necessarily dominant.

\subsection{Maximal multiplicity and quasipolynomial factorization}
\label{SecMaxMultAndFact}

The main result of this section is the following characterization of the existence of a real root attaining the maximal multiplicity $\mathscr D_{PS}$.

\begin{theorem}
\label{MaxMult}
Consider the quasipolynomial $\Delta$ given by \eqref{Delta} and let $s_0 \in \mathbb R$.
 The number $s_0$ is a root of multiplicity $\mathscr{D}_{PS}=m+n+1$ of $\Delta$ if and only if
\begin{equation}
\label{Coeffs}
\left\{
\begin{aligned}
a_k & = \binom{n}{k} \left( -s_{{0}} \right) ^{n-k}+
 \left( -1 \right) ^{n-k}n!\,\sum _{j=k}^{n-1}{\frac {\binom{j}{k} \binom{m+n
-j}{m} s_0^{j-k}}{j!\,{\tau}^{n-j}}} & \quad & \textnormal{for every } k \in \llbracket 0, n-1\rrbracket, \\
\alpha_k & =\left( 
-1 \right) ^{n-1}{{e}^{s_{{0}}\tau}}\sum _{j=k}^{m}{\frac {
 \left( -1 \right) ^{j-k} \left( m+n-j \right) !\,s_0^{j-k
}}{k!\, \left( j-k \right) !\, \left( m-j \right) !\,{\tau}^{n-j}}} & & \textnormal{for every } k \in \llbracket 0, m\rrbracket.
\end{aligned}
\right.
\end{equation}
\end{theorem}

Before turning to the proof of Theorem~\ref{MaxMult}, a few remarks are in order.

\begin{remark}
By considering the first equation in \eqref{Coeffs} with $k = n-1$, one obtains the simple and interesting relation between $s_0$, $\tau$, and $a_{n-1}$ given by
\begin{equation}
\label{RelationS0ANMinus1Tau}
s_0 = -\frac{a_{n-1}}{n} - \frac{m+1}{\tau}.
\end{equation}
\end{remark}

\begin{remark}
\label{REM-1REV}
Under conditions \eqref{Coeffs}, $s_0$ is the unique real root of $\Delta$, and, more precisely, it is the unique root of $\Delta$ on the horizontal strip $\{s \in \mathbb C \suchthat \abs{\Im(s)} < \frac{2\pi}{\tau}\}$ of the complex plane. Indeed, applying Proposition~\ref{PropPolyaSzego} with $\alpha = 0$ and any $\beta \in (0, \frac{2\pi}{\tau})$, we deduce that the number of roots of $\Delta$ in the strip $S_\beta = \{s \in \mathbb C \suchthat 0 \leq \Im(s) \leq \beta\}$, counted according to their multiplicity, is at most $\mathscr D_{PS}$, since $\frac{\tau(\beta - \alpha)}{2 \pi} < 1$ in this case. Since $s_0 \in S_\beta$ and its multiplicity as a root of $\Delta$ is exactly $\mathscr D_{PS}$, this shows that no other root of $\Delta$ can belong to $S_\beta$. The conclusion now follows since $\beta \in (0, \frac{2\pi}{\tau})$ is arbitrary, and by remarking that $\Delta$ has only real coefficients and thus its roots appear in complex-conjugate pairs.
\end{remark}

\begin{remark}
\label{REM-2REV}
Theorem~\ref{MaxMult} deals only with a \emph{real} root $s_0$ of $\Delta$ attaining the maximal multiplicity $\mathscr D_{PS}$. The main reason for this restriction is that, as shown in \cite[Corollaries~1 and 2]{Boussaada2016Tracking}, nonreal roots of $\Delta$ cannot have a multiplicity equal to the P\'{o}lya--Szeg\H{o} bound $\mathscr D_{PS}$ (see also \cite{mazantiComplex} for further discussion on the maximal multiplicity of complex roots in the case $n = 2$ and $m = 1$). Indeed, any root $s_0$ of $\Delta$ attaining the maximal multiplicity $\mathscr D_{PS}$ necessarily satisfies \eqref{RelationS0ANMinus1Tau}, and thus it will be real since $a_{n-1}$ is real.
\end{remark}

The proof of Theorem~\ref{MaxMult} can be simplified by noticing that $s_0$ is a root of the quasipolynomial $\Delta$ of a given multiplicity if and only if $0$ is a root with the same multiplicity of the quasipolynomial $\widetilde\Delta$ defined by $\widetilde\Delta(z) = \tau^n \Delta(s_0 + \frac{z}{\tau})$. This transformation of $\Delta$ into $\widetilde\Delta$ corresponds to the linear change of variable $z = \tau(s - s_0)$, which shifts the desired multiple root $s_0$ to $0$ and reduces the delay $\tau$ to $1$, and it has been described in Lemmas~\ref{LemmDeltaTilde} and \ref{LemmCoeffsInverseTransform} in Section~\ref{SecQuasipoly}. We can then focus on providing necessary and sufficient conditions on the coefficients of $\widetilde\Delta$ in order for $0$ to be a root of maximal multiplicity $\mathscr{D}_{PS}=m+n+1$, which is the main subject of our next result, Lemma~\ref{LemmCoeffsBetaB}, which also states that, under such conditions, $\widetilde\Delta$ can be factorized as the product of $z^{m+n+1}$ and an entire function expressed as the Laplace transform of a Bernstein polynomial, as the one discussed in Example~\ref{LaplaceBernstein}.

\begin{lemma}
\label{LemmCoeffsBetaB}
Let $n \in \mathbb N^\ast$ and $m \in \mathbb N$ satisfy $m\leq n$, $b_0, \dotsc, b_{n-1}, \beta_0, \dotsc, \beta_{m} \in \mathbb R$, and $\widetilde\Delta$ be the quasipolynomial given by \eqref{DeltaTilde}. Then $0$ is a root of multiplicity $\mathscr{D}_{PS} = m + n + 1$ of $\widetilde\Delta$ if and only if
\begin{equation}
\label{BetaBMaxMultiplicity}
\left\{
\begin{aligned}
b_k & = (-1)^{n-k} \frac{n!}{k!} \binom{m+n-k}{m} & \quad & \textnormal{for every } k \in \llbracket 0, n-1\rrbracket, \\
\beta_k & = (-1)^{n-1} \frac{(m+n-k)!}{k! (m-k)!} & & \textnormal{for every } k \in \llbracket 0, m\rrbracket.
\end{aligned}
\right.
\end{equation}

Moreover, if \eqref{BetaBMaxMultiplicity} is satisfied, then, for every $z \in \mathbb C$,
\begin{equation}
\label{FactorizationWidetildeDelta}
\widetilde\Delta(z) = \frac{z^{m+n+1}}{m!} \int_0^1 t^{m} (1-t)^{n} e^{-zt} \diff t.
\end{equation}
\end{lemma}

\begin{proof}
We first prove that the right-hand side of \eqref{FactorizationWidetildeDelta} is indeed a quasipolynomial of the form \eqref{DeltaTilde} with coefficients $b_0, \dotsc, b_{n-1}, \beta_0, \dotsc, \beta_m$ given by \eqref{BetaBMaxMultiplicity} and admitting $0$ as a root of multiplicity $\mathscr{D}_{PS}$. For that purpose, let us introduce the function $Q: \mathbb C \to \mathbb C$ defined as the right-hand side of \eqref{FactorizationWidetildeDelta}, i.e., for every $z \in \mathbb C$,
\[
Q(z) = \frac{z^{m+n+1}}{m!} \int_0^1 t^{m} (1-t)^{n} e^{-zt} \diff t.
\]
Clearly, $Q$ is an entire function and $0$ is a root of multiplicity $m + n + 1$ of $Q$. Let $p: \mathbb R \to \mathbb R$ be the Bernstein polynomial defined for $t \in \mathbb R$ by $p(t) = t^m (1-t)^n$, whose degree is $m + n$. Then, using Proposition~\ref{PropIntegralPExp}, we have
\begin{equation}
\label{ExpansionQ}
Q(z) = \frac{z^{m+n+1}}{m!} \sum_{k = 0}^{m + n} \frac{p^{(k)}(0) - p^{(k)}(1) e^{-z}}{z^{k + 1}}.
\end{equation}
We have
\[
p(t) = t^m \sum_{k = 0}^n \binom{n}{k} (-1)^k t^k = \sum_{k=m}^{m + n} (-1)^{k - m} \binom{n}{k - m} t^k,
\]
and thus $p^{(k)}(0) = 0$ for $k \in \llbracket 0, \dotsc, m-1\rrbracket$ and $p^{(k)}(0) = (-1)^{k - m} k! \binom{n}{k - m}$ for $k \in \llbracket m, m + n\rrbracket$. Similarly, we have
\[
p(t) = (1 - t)^n (t - 1 + 1)^m = (-1)^n (t-1)^n \sum_{k=0}^m \binom{m}{k} (t - 1)^k = (-1)^n \sum_{k=n}^{m + n} \binom{m}{k - n} (t - 1)^{k},
\]
and thus $p^{(k)}(1) = 0$ for $k \in \llbracket 0, n - 1\rrbracket$ and $p^{(k)}(1) = (-1)^n k! \binom{m}{k - n}$ for $k \in \llbracket n, m + n\rrbracket$. Thus, \eqref{ExpansionQ} becomes
\begin{align*}
Q(z) & = \sum_{k = m}^{m + n} (-1)^{k - m} \frac{k!}{m!} \binom{n}{k - m} z^{m + n - k} + e^{-z} \sum_{k = n}^{m + n} (-1)^{n - 1} \frac{k!}{m!} \binom{m}{k - n} z^{m + n - k} \\
& = \sum_{k = 0}^{n} (-1)^{n - k} \frac{(m + n - k)!}{m!} \binom{n}{k} z^{k} + e^{-z} \sum_{k = 0}^{m} (-1)^{n - 1} \frac{(m + n - k)!}{m!} \binom{m}{k} z^{k}.
\end{align*}
Hence, $Q$ is indeed a quasipolynomial under the form \eqref{DeltaTilde} with coefficients $b_0, \dotsc, b_{n-1},\allowbreak \beta_0, \dotsc, \beta_m$ given by \eqref{BetaBMaxMultiplicity} and admitting $0$ as a root of multiplicity $\mathscr{D}_{PS}$, as required.

In order to conclude the proof, it suffices to show that \eqref{BetaBMaxMultiplicity} is the unique choice of coefficients for $\widetilde\Delta$ ensuring that $0$ is a root of multiplicity $m + n + 1$. To see that, notice that, since the degree of the quasipolynomial $\widetilde\Delta$ is $m + n + 1$, $0$ is a root of multiplicity $m + n + 1$ of $\widetilde\Delta$ if and only if $\widetilde\Delta^{(k)}(0) = 0$ for every $k \in \llbracket 0, m+n\rrbracket$. By \cite[Proposition~5.1]{Boussaada2016Characterizing}, this is the case if and only if the coefficients $b_0, \dotsc, b_{n-1}, \beta_0, \dotsc, \beta_{m}$ satisfy
\begin{equation}
\label{eq:syst-b-beta}
\left\{
\begin{aligned}
b_k & = - \beta_k - \sum_{\ell = 0}^{k-1} \frac{(-1)^{k-\ell} \beta_\ell}{(k - \ell)!}, & \qquad & \forall k \in \llbracket 0, \min(m, n-1)\rrbracket, \\
b_k & = - \sum_{\ell = 0}^{m} \frac{(-1)^{k-\ell} \beta_\ell}{(k - \ell)!}, & \qquad & \forall k \in \llbracket m+1, n-1\rrbracket, \\
1 & = - \sum_{\ell = 0}^{m} \frac{(-1)^{n - \ell} \beta_\ell}{(n - \ell)!}, & & \\
0 & = - \sum_{\ell = 0}^{m} \frac{(-1)^{k - \ell} \beta_\ell}{(k - \ell)!}, & \qquad & \forall k \in \llbracket n + 1, m+n\rrbracket.
\end{aligned}
\right.
\end{equation}
The first two equalities in \eqref{eq:syst-b-beta} allow us to compute $b_0, \dotsc, b_{n-1}$ once $\beta_0, \dotsc, \beta_m$ are known, and the last two equalities in \eqref{eq:syst-b-beta} provide a linear system on $\beta_0, \dotsc, \beta_m$. This linear system can be written under the form $T \beta = -e_0$, where $T = (T_{j, k})_{j, k \in \llbracket 0, m\rrbracket}$ is defined by
\[
T_{j, k} = \frac{(-1)^{n + j - k}}{(n + j - k)!}, \qquad j, k \in \llbracket 0, m\rrbracket,
\]
$\beta = (\beta_0, \dotsc, \beta_m)^{\transp}$, and $e_0 = (1, 0, \dotsc, 0)^{\transp} \in \mathbb R^{m+1}$. By the previous arguments, the coefficients $b_0, \dots, b_{n-1}, \beta_0, \dotsc, \beta_m$ given by \eqref{BetaBMaxMultiplicity} necessarily satisfy \eqref{eq:syst-b-beta}, since they ensure that $0$ is a root of multiplicity $m + n + 1$ of $\widetilde\Delta$, and thus we are left to prove that \eqref{eq:syst-b-beta} admits a unique solution, which is equivalent to showing that the matrix $T$ is invertible.

Let us introduce the diagonal matrices $D_0 = \diag(0!, 1!, \dotsc, m!)$ and $D_n = \diag(n!, (n+1)!, \dotsc, ( m+n)!)$ and set $B = D_n T D_0^{-1}$. Writing $B = (B_{j, k})_{j, k \in \llbracket 0, m\rrbracket}$, we have
\[
B_{j, k} = (-1)^{n + j - k} \binom{n + j}{k}, \qquad j, k \in \llbracket 0, m\rrbracket.
\]
We have $B = L U$, where $L = (L_{j, k})_{j, k \in \llbracket 0, m\rrbracket}$ and $U = (U_{j, k})_{j, k \in \llbracket 0, m\rrbracket}$ are defined, for $j, k \in \llbracket 0, m\rrbracket$, by $L_{j, k} = (-1)^{j - k} \binom{j}{k}$ and $U_{j, k} = (-1)^{n+j-k}\binom{n}{k-j}$. Indeed, for $j, k \in \llbracket 0, m\rrbracket$, we have
\begin{equation*}
\sum_{\ell = 0}^{m} L_{j, \ell} U_{\ell, k} = (-1)^{n+j-k} \sum_{\ell = 0}^{m} \binom{j}{\ell} \binom{n}{k-\ell} = (-1)^{n+j-k} \binom{n+j}{k} = B_{j, k},
\end{equation*}
where we use the identity $\sum_{\ell = 0}^m \binom{j}{\ell} \binom{n}{k - \ell} = \binom{n + j}{k}$ (see, e.g., \cite[Proposition~2.16]{MBN-2021-JDE}). Notice that $L$ is lower triangular and $U$ is upper triangular, and hence the factorization $B = L U$ corresponds to the LU factorization of $B$. As a consequence of this factorization, we also deduce that $\det B = (-1)^{n\,(m+1)}$ and, in particular, $B$ is invertible. Hence, $T = D_n^{-1} B D_0$ is also invertible, concluding the proof.
\end{proof}

\begin{proof}[Proof of Theorem \ref{MaxMult}]
Define $\widetilde\Delta$ from $\Delta$ as in \eqref{DefiDeltaTilde}. One immediately verifies that $s_0$ is a root of multiplicity $m+n+1$ of $\Delta$ if and only if $0$ is a root of multiplicity $m+n+1$ of $\widetilde\Delta$. The result then follows by a straightforward computation using Lemmas \ref{LemmCoeffsInverseTransform} and \ref{LemmCoeffsBetaB}.
\end{proof}

\subsection{Dominance of roots of maximal multiplicity}

Theorem~\ref{MaxMult} provides necessary and sufficient conditions for a real number $s_0$ to be a root of maximal multiplicity of the quasipolynomial $\Delta$ from \eqref{Delta}. The main result of this section states that, under those conditions, $s_0$ is necessarily a dominant root of $\Delta$.

\begin{theorem}
\label{Dominance}
Consider the quasipolynomial $\Delta$ given by \eqref{Delta} and let $s_0 \in \mathbb R$.
\begin{enumerate}
\item\label{ItemB}(Retarded) If $m<n$ and  \eqref{Coeffs} is satisfied, then $s_0$ is a strictly dominant root of $\Delta$.
\item\label{ItemD}(Neutral) If $m=n$ and  \eqref{Coeffs} is satisfied then, $s_0$ is a dominant root of $\Delta$ and, for every other complex root $s$ of $\Delta$, one has $\Re(s) = s_0$. More precisely, the set of roots of $\Delta$ is $\{s_0 + i \frac{\zeta}{\tau} \suchthat \zeta \in \Xi_n\}$, where
\begin{equation}\label{NeutralImaginaryCluster}
\Xi_n = \left\{\zeta \in \mathbb R \suchthat \tan\left(\frac{\zeta}{2}\right) = \frac{\displaystyle \zeta \sum_{\ell = 0}^{\floor*{\frac{n-1}{2}}} (-1)^\ell \frac{(2n - 2\ell - 1)!}{(2\ell + 1)! (n - 2\ell - 1)!} \zeta^{2\ell}}{\displaystyle \sum_{\ell = 0}^{\floor*{\frac{n}{2}}} (-1)^\ell \frac{(2n - 2\ell)!}{(2\ell)! (n - 2\ell)!} \zeta^{2\ell}}\right\}.
\end{equation}
\end{enumerate}
\end{theorem}

\begin{remark}
In the particular case of $n=m=1$, the spectrum distribution described by \eqref{NeutralImaginaryCluster}, i.e., $\Xi_1 = \left\{\zeta \in \mathbb R \suchthat \tan\left(\frac{\zeta}{2}\right) = \frac{\zeta}{2}\right\}$, has been already identified in \cite{MBNC-2021-MTNS}.
\end{remark}

To prove Theorem~\ref{Dominance}, we rely, as in the proof of Theorem~\ref{MaxMult}, on the linear change of variable $z = \tau(s - s_0)$. We thus first study the quasipolynomial $\widetilde\Delta$ from \eqref{DeltaTilde} under conditions \eqref{BetaBMaxMultiplicity}, in which case we have, by Theorem~\ref{MaxMult}, the factorization \eqref{FactorizationWidetildeDelta}. The factorization \eqref{FactorizationWidetildeDelta} can also be written, thanks to \eqref{EqKummerIntegral}, as
\begin{equation}
\label{FactorizationKummer}
\widetilde\Delta(z) = \frac{n!}{(m + n + 1)!} z^{m+n+1} \Phi(m+1, m+n+2, -z),
\end{equation}
where $\Phi$ is the Kummer confluent hypergeometric function defined in \eqref{DefiConfluent}. The next lemma uses properties of the roots of $\Phi$ in order to deduce that $0$ is a dominant root of $\widetilde\Delta$.

\begin{lemma}
\label{LemmDominancy}
Let $n \in \mathbb N^\ast$, $b_0, \dotsc, b_{n-1}, \beta_0, \dotsc, \beta_{m} \in \mathbb R$ be given by \eqref{BetaBMaxMultiplicity}, and $\widetilde\Delta$ be the quasipolynomial given by \eqref{DeltaTilde}. Let $z$ be a root of $\widetilde\Delta$ with $z \neq 0$. Then $\Re(z) < 0$ if $m<n$ and $\Re(z) = 0$ if $m=n$.
\end{lemma}

\begin{proof}
By Lemma~\ref{LemmCoeffsBetaB}, $\widetilde\Delta$ admits the factorization \eqref{FactorizationKummer}. Hence, if $z$ is a root of $\widetilde\Delta$ with $z \neq 0$, then $-z$ must be a root of $\Phi(m+1, m+n+2, \cdot)$. It follows from Proposition~\ref{CorZerosKummer}\ref{CorZerosKummer-k-geq-0} that, for $n>m$, one has $\Re(-z) > 0$, and thus $\Re(z) < 0$. Also, by applying Proposition~\ref{CorZerosKummer}\ref{CorZerosKummer-k-eq-0}, we deduce that, for $n=m$, one has $\Re(-z) = 0$, and thus $\Re(z) = 0$.
\end{proof}

In order to obtain the additional characterization of the spectrum in the neutral case stated in Theorem~\ref{Dominance}\ref{ItemD}, we also provide the following result on the location of the roots of $\widetilde\Delta$ in the case $m = n$.

\begin{lemma}
\label{LemmImaginaryRoots}
Let $n \in \mathbb N^\ast$, $m = n$, $b_0, \dotsc, b_{n-1}, \beta_0, \dotsc, \beta_{n} \in \mathbb R$ be given by \eqref{BetaBMaxMultiplicity}, $\widetilde\Delta$ be the quasipolynomial given by \eqref{DeltaTilde}, $\Xi_n$ be as in the statement of Theorem~\ref{Dominance}\ref{ItemD}, and $\zeta \in \mathbb R$. Then $i\zeta$ is a root of $\widetilde\Delta$ if and only if $\zeta \in \Xi_n$.
\end{lemma}

\begin{proof}
From \eqref{DeltaTilde} and \eqref{BetaBMaxMultiplicity}, we have
\[
\widetilde\Delta(z) = \sum_{k=0}^n (-1)^{n - k} \frac{(2n - k)!}{k! (n - k)!} z^k + e^{-z} \sum_{k = 0}^n (-1)^{n - 1} \frac{(2n - k)!}{k! (n - k)!} z^k.
\]
Hence, $z = i \zeta$ is a root of $\widetilde\Delta$ if and only if
\[
e^{i \frac{\zeta}{2}}\sum_{k=0}^n (-1)^{k} \frac{(2n - k)!}{k! (n - k)!} (i\zeta)^k = e^{-i\frac{\zeta}{2}} \sum_{k = 0}^n \frac{(2n - k)!}{k! (n - k)!} (i\zeta)^k.
\]
The right-hand side of the above equality is equal to the complex conjugate of the left-hand side. Hence, the above equality is equivalent to
\[
\Im\left(e^{i \frac{\zeta}{2}}\sum_{k=0}^n (-1)^{k} \frac{(2n - k)!}{k! (n - k)!} (i\zeta)^k\right) = 0,
\]
which happens if and only if
\[
\sin\left(\frac{\zeta}{2}\right) \sum_{\ell = 0}^{\floor*{\frac{n}{2}}} (-1)^\ell \frac{(2n - 2\ell)!}{(2\ell)! (n - 2\ell)!} \zeta^{2\ell} = \zeta \cos\left(\frac{\zeta}{2}\right) \sum_{\ell = 0}^{\floor*{\frac{n-1}{2}}} (-1)^\ell \frac{(2n - 2\ell - 1)!}{(2\ell + 1)! (n - 2\ell - 1)!} \zeta^{2\ell}.
\]
This last equality is equivalent to $\zeta \in \Xi_n$, concluding the proof.
\end{proof}

Using Lemmas~\ref{LemmDominancy} and \ref{LemmImaginaryRoots}, we can easily conclude the proof of Theorem~\ref{Dominance}.

\begin{proof}[Proof of Theorem~\ref{Dominance}]
Define $\widetilde\Delta$ from $\Delta$ as in \eqref{DefiDeltaTilde}. For $n>m$, item \ref{ItemB} can be shown by noticing that, if $s$ is a root of $\Delta$ with $s \neq s_0$, then, by \eqref{DefiDeltaTilde}, $z = \tau(s - s_0)$ is a root of $\Delta$ with $z \neq 0$. Hence, by Lemma \ref{LemmDominancy}, $\Re(\tau(s - s_0)) < 0$, showing, since $\tau > 0$, that $\Re(s) < s_0$. Similarly, for $m=n$, the first part of item \ref{ItemD} can be shown by applying first Lemma~\ref{LemmDominancy}, which gives $\Re(s) = s_0$ for any root $s$ of $\Delta$, and the last part of item \ref{ItemD} follows by applying Lemma~\ref{LemmImaginaryRoots}.
\end{proof}

\subsection{Consequences on stability}

The third main result we present in this paper is the following one on the the stability of the trivial solution of \eqref{MainSystTime}, which is an immediate consequence of \eqref{RelationS0ANMinus1Tau}, Theorem~\ref{MaxMult}, and Theorem~\ref{Dominance}.

\begin{theorem}
\label{ExpStab}
Let $n \in \mathbb N^\ast$, $m \in \llbracket 0, n\rrbracket$, $\tau > 0$, and consider the delay-differential equation \eqref{MainSystTime}. Assume that \eqref{Coeffs} is satisfied for some $s_0 \in \mathbb R$. Then the trivial solution of \eqref{MainSystTime} is exponentially stable if and only if $a_{n-1}>-\frac{n\,(m+1)}{\tau}$.
\end{theorem}

\section{Illustrative examples}
\label{SecIllustrative}

PID controllers have been extensively used to control and regulate industrial processes which are typically modeled by reduced-order dynamics. In this section, we shall illustrate how one can tune the controller gains using the GMID property through a couple of comprehensive examples. The first case deals with a retarded equation, while the second one corresponds to a neutral equation.

\subsection{Stabilizing the classical pendulum via delay action}
 Consider the dynamical system modeling a friction-free classical pendulum. The adopted model is studied in \cite{A1999} and, in the sequel, we keep the same notations.
The dynamics of a controlled classical pendulum are governed by the following second-order differential equation
\begin{equation}\label{InvPenNLin}
\ddot \theta(t) + \frac{g}{L}\sin(\theta(t)) = u(t)
\end{equation}
where $\theta(t)$ stands for the angular displacement of the pendulum at time $t$ with respect to the stable equilibrium position, $L$ is the length of the pendulum, $g$ is the gravitational acceleration, and $u(t)$ represents the control law, which stems from an applied external torque.
The problem we consider is to make the open-loop stable equilibrium locally asymptotically stable via the action of the external torque $u$. This suggests first to
linearize \eqref{InvPenNLin} around the trivial equilibrium and get 
\begin{equation}\label{InvPenLin}
\ddot \theta(t) + \frac{g}{L}\,\theta(t) = u(t).
\end{equation}
Next assume that such a system is controlled using a standard delayed PD controller of the form
\begin{equation}\label{DPD}
u(t)=-k_p\,\theta(t-\tau)-k_d\,\dot\theta(t-\tau),
\end{equation}
 with $(k_p,k_d)\in\mathbb{R}^2$. The local stability of the closed-loop system is reduced to study the location of the spectrum of the quasipolynomial 
\begin{equation}
\label{DeltaPendulum}
    \Delta(s)=s^{2}+\frac{g}{L}+\left(k_{d} s +k_{p}\right) {e}^{-\tau  s},
\end{equation}
which is a quasipolynomial of degree 4. Remarks~\ref{REM-1REV}--\ref{REM-2REV} allow concluding that the maximal multiplicity 4 can be reached only by a root of $\Delta$ on the real axis.

In this case, we apply the GMID property to establish the gains of the stabilizing delayed PD controller. As a matter of fact, in this case,  the only admissible quadruple root is given by
\begin{equation}\label{IPQR}
s_0 =  -\frac{\sqrt{2}}{\sqrt{\frac{L}{g}}}, 
\end{equation}
which is achieved if the controllers' gains and the delay $\tau$ are taken as
\begin{equation*}
    k_{d}= 
-\frac{{e}^{-2} \sqrt{2}}{\sqrt{\frac{L}{g}}},\quad k_{p} = 
-\frac{5 \,{e}^{-2} g}{L},\quad\tau = 
\sqrt{2}\, \sqrt{\frac{L}{g}}.
\end{equation*}

With the proposed choice of the gains and the delay, conditions \eqref{Coeffs} are satisfied for the quasipolynomial \eqref{DeltaPendulum}, and thus, by Theorem~\ref{Dominance}, $s_0$ given by \eqref{IPQR} is the spectral abscissa of the closed-loop system \eqref{InvPenLin}--\eqref{DPD}, ensuring then the exponential stability of the trivial solution as announced in Theorem~\ref{ExpStab}.

\begin{remark}
Notice that the GMID consists in  forcing a root to reach its \emph{maximal} multiplicity, which does not allow any degree of freedom in assigning $s_0$, as precised in \eqref{IPQR}. In order to allow for some additional freedom when assigning $s_0$, one can relax such a constraint by forcing the root $s_0$ to have a multiplicity lower than the maximal, and also consider the delay as a free tuning parameter. This motivates the study of a (non-generic) MID property, which was carried out, for instance, in \cite{Boussaada2020Multiplicity} for second-order systems.

\begin{figure}[ht]
\centering
 \includegraphics[width=0.5\textwidth]{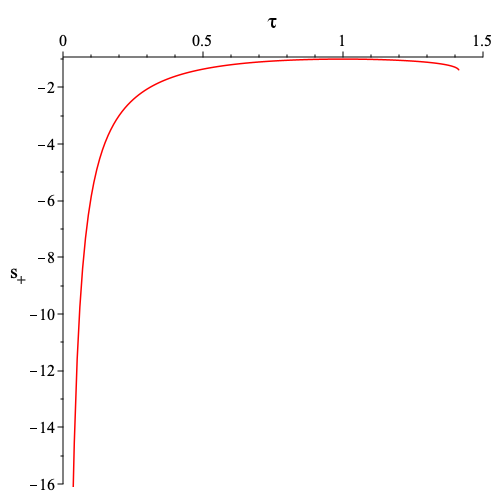}
\caption{The behavior of the triple root (spectral abscissa) of \eqref{DeltaPendulum} at $s=s_+$ as a function of the free delay parameter $0<\tau<\sqrt{2\,L/g}$ for $L/g=1$.}
\label{Pendulum-TRIPLE}
\end{figure}

For instance, the only admissible triple roots of \eqref{DeltaPendulum} are given by
\begin{equation*}
s_\pm =  \frac{-2\pm\sqrt{-\frac{g \,\tau^{2}}{L}+2}}{\tau},
\end{equation*}
which are achieved if $\tau < \sqrt{{2\,L}/{g}}$ and the controllers' gains are taken as
\begin{equation*}
k_{d} = 
\frac{2 \left(\tau  s_{\pm}+1\right) {\mathrm e}^{\tau  s_{\pm}}}{\tau}, \qquad k_{p} = 
\frac{2 \left(g \,\tau^{2}/L+5 \tau  s_{\pm}+3\right) {\mathrm e}^{\tau  s_{\pm}}}{\tau^{2}}.
\end{equation*}
Following \cite[Theorem 4.2]{Boussaada2020Multiplicity}, the triple root at $s=s_+$ is the rightmost root of \eqref{DeltaPendulum}. Thus the delay, if seen as a tuning  parameter, allows to assign the rightmost root at $s=s_+$ arbitrarily large (in absolute value) for small delay, as illustrated in Figure \ref{Pendulum-TRIPLE}.
\end{remark}
 
\subsection{Feedback stabilization for a scalar conservation law
with PI boundary control}

 As an illustrative example of the GMID for neutral equations, we consider the problem of stabilization of solutions of a partial differential  equation of hyperbolic type.
 More precisely, we revisit the problem of exponential stabilization of the following scalar conservation law proposed in \cite{Coron2015Feedback}:
\begin{equation}\label{ConservLaw}
\partial_t \varphi(t,x)+\lambda\,\partial_x \varphi(t,x)=0, \quad t\in[0,\,\infty), \quad x\in(0,\,L),
\end{equation}
where $L > 0$ and $\varphi(t, x)$ denotes the system state at position $x \in (0, L)$ and in time $t \in [0,+\infty)$. As considered in \cite{Coron2015Feedback}, the value $\lambda$, which denotes the velocity of propagation, is assumed to be a positive constant.
Equation \eqref{ConservLaw} is accompanied with a  boundary condition under the form of a PI controller:
\begin{equation}\label{BC}
    \varphi(t,\,0)=k_p\,\varphi(t,L)+k_i\,\int_0^t \varphi(\nu,L) \diff \nu,
\end{equation}
where $k_p$ and $k_i$ are the feedback parameters representing proportional and integral control gains. Applying the Laplace transform to both sides of the boundary condition and multiplying by $s$ one obtains the closed-loop characteristic function
\begin{equation}\label{CEConLaw}
\Delta(s)=s-(k_i+k_p\,s)e^{-\frac{L}{\lambda}\,s},
\end{equation}
which corresponds to the characteristic function of a first-order neutral equation, that is, a function under the form \eqref{Delta} with $m=n=1$. In this case, the degree $\mathscr{D}_{PS}$ of $\Delta$ is equal to $3$ and, as mentioned in Remarks~\ref{REM-1REV}--\ref{REM-2REV}, the maximal multiplicity can be achieved only by a root on the real axis.

Next, by exploiting the results of Theorem~\ref{MaxMult}, Theorem~\ref{Dominance}, and Theorem~\ref{ExpStab}, we conclude that forcing a triple spectral value guarantees its dominance as a root of \eqref{CEConLaw}, and then the exponential stability of solutions of \eqref{ConservLaw}--\eqref{BC}. More precisely, by tuning the controller gains as
\begin{equation}
\label{CEkpki}
k_{p} =- {e}^{-2},\;\quad k_{i} =- \frac{4 \,{e}^{-2}\,\lambda}{L},
\end{equation} one achieves the unique admissible triple root, which is $s_0= -\frac{2\,\lambda}{L}$ and corresponds to the decay rate of solutions of \eqref{ConservLaw}--\eqref{BC}. Furthermore, as shown in Theorem~\ref{Dominance}\ref{ItemD}, the set of roots of $\Delta$ is $\left\{s_0 + i \frac{\lambda\,\zeta}{L} \suchthat \zeta \in \Xi_1\right\}$ where  $\Xi_1 = \left\{\zeta \in \mathbb R \suchthat \tan\left(\frac{\zeta}{2}\right) = \frac{\zeta}{2}\right\}$. Figure~\ref{FConsLaw}(a) shows the result of a numerical computation of the roots of \eqref{CEConLaw} with the parameters \eqref{CEkpki}, while Figure~\ref{FConsLaw}(b) shows the solution of \eqref{ConservLaw}--\eqref{BC} in the case $\frac{L}{\lambda} = 1$ with an initial condition $\varphi(0, x) = \sin(2 \pi x)$.

\begin{figure}[ht]
\centering
\begin{tabular}{@{} >{\centering} m{0.5\textwidth} @{} >{\centering} m{0.5\textwidth} @{}}
\resizebox{0.5\textwidth}{!}{\input{Figures/transport_eqn_spectrum.pgf}} & \includegraphics[width=0.5\textwidth]{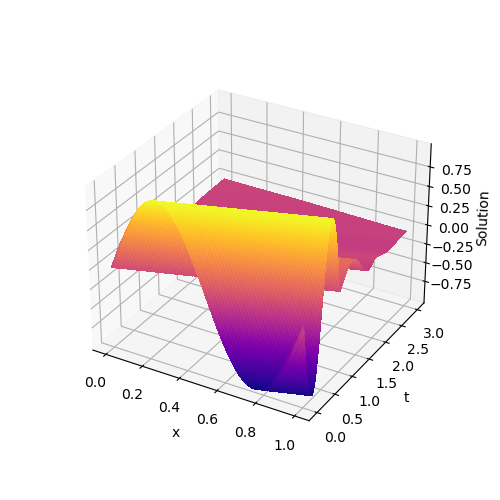} \tabularnewline
(a) & (b) \tabularnewline
\end{tabular}
\caption{(a) Spectrum distribution of \eqref{CEConLaw} and (b) solution of \eqref{ConservLaw} with initial condition $\varphi(0, x) = \sin(2 \pi x)$, with $\frac{L}{\lambda} = 1$ and parameters $k_p$ and $k_i$ satisfying \eqref{CEkpki}.}
\label{FConsLaw}
\end{figure}

\subsection{New features of P3\texorpdfstring{$\delta$}{delta}}

P3$\delta$, whose name stands for \emph{Partial pole placement via delay action}, is a Python software with a friendly user interface for the design of parametric stabilizing feedback laws with time delays, thanks to properties of the distribution of quasipolynomials' zeros. It exploits mainly the MID/GMID property as well as the coexisting real roots-induced-dominancy or CRRID for short, established in \cite{BoussaadaPartial, BedouheneReal ,Amrane2018Qualitative}. Initially, P3$\delta$ has been established for the control design for retarded differential equations. Based on the results of this work as well as the references \cite{Ma,Benarab2020MID,MBNC-2021-MTNS}, the software has been updated and is now able to treat linear neutral functional differential equations as well.

P3$\delta$ is freely available for download on \url{https://cutt.ly/p3delta}, where installation instructions, video demonstrations, and the user guide are also available. Interested readers may also contact directly any of the authors of the paper. Since its creation, P3$\delta$ had vocation to be available to the greatest number and on all possible platforms. The current version of the software is available in local executable version as well as an online version ready to use in one click. The online version of P3$\delta$ is hosted on servers thanks to the \emph{Binder} service \cite{project_jupyter-proc-scipy-2018}. \emph{Binder} allows to create instances of personalized computing environment directly from a \emph{GitHub} repository
that can be employed and shared by users. The \emph{Binder} service is free to use and is powered by \emph{BinderHub}, an open-source tool that deploys the service in the cloud. The online version of P3$\delta$ is written in Python 
and structured as a \emph{Jupyter Notebook}, an open document format which can contain live code, equations, visualizations, and text.
The Jupyter Notebook is completed by a friendly user interface built using interactive widgets from Python's \texttt{ipywidgets} module.

\subsection{Further remarks on the maximal damping}

Although there exists a root locus approach for single delay time-delay systems \cite{Krall1970} able to monitor the variation of the spectrum with respect to parameters variation, then  allowing to characterize the spectral abscissa, unfortunately this method  doesn't recover the multiple root case due to the lack of regularity with respect to parameter's variation.

To the best of the authors' knowledge, the problem of  the selection of the system's parameters  guaranteeing the maximum damping of reduced-order time-delay systems solutions goes back to the almost forgotten contributions of Pinney in the 1960s \cite{Pinney1958Ordinary}, where a complex analysis method based on Cauchy index theorem coupled with the D-partition method \cite{Neimark1948Structure} are used. Furthermore, the link between a spectral value of maximal multiplicity and the maximum damping has been emphasized  for first- and second-order DDEs, however without providing explicit proofs. As a matter of fact, it has been observed that the minimization of the spectral abscissa suggests the ``exhaustion'' of the equation parameters, which is achieved in Equation \eqref{Coeffs} from Theorem~\ref{MaxMult}.

Next, in the context of the design of output feedback able to achieve a fast stabilization of finite-dimensional systems,  an optimization technique called \emph{Convex hull technique} has been proposed in \cite{Chen1979Output}, where the existing link between the maximum damping and the characteristic root of maximum multiplicity has been established, see also \cite{Blondel2012Explicit}. More recently, to recover such a link for second-order delay systems, an optimization approach has been used in  \cite{Ramirez2016Design}. Notice that the extension of such optimization methods to neutral DDEs or even to retarded DDEs of fixed but higher orders (greater than 3) systems is a quite complicated question especially in the presence of free parameters.

\section{Concluding remarks}
In this paper, we have further investigated the \emph{multiplicity-induced-dominancy property\/}  for single-delay linear functional differential equations. Thanks to the reduction of the corresponding characteristic function to an appropriate integral representation that corresponds to some degenerate Kummer hypergeometric function,  we have shown the validity of the \emph{generic multiplicity-induced-dominancy (GMID) property\/};  that is, the characteristic spectral values of maximal multiplicity are necessarily dominant for retarded as well as neutral delay-differential equations of arbitrary order. 

Notice that the MID  property may hold even when it is about a spectral value with a strictly-intermediate admissible multiplicity, as shown in \cite{Boussaada2020Multiplicity} for second-order plants and in \cite{Balogh21}  for $n$\textsuperscript{th}-order retarded equations admitting a real spectral value with multiplicity $n+1$ and a  finite dimensional part admitting exclusively real modes. However, in all generality, the limits of the MID property remain an open question.  

All these arguments make our constructive method based on confluent hypergeometric functions applied to linear delay-differential equations (retarded as well as neutral) of arbitrary order particularly attractive and advantageous in many ways and, in particular, in control design problems involving free parameters to be tuned. As shown through some illustrative examples concerning classical PID control schemes, the partial pole placement method that is based on the use of GMID property may appear as a good alternative to the existing control methodologies. Indeed, the corresponding controllers are of ``low-complexity'' (i.e., reduced number of parameters to be tuned) and there exist explicit guarantees on the location of the remaining characteristic roots if the GMID is reached.

\section*{Acknowledgments}
The authors warmly thank the Editor and the Reviewer for  their comments and suggestions which contributed to improve the overall quality of the paper.\\
This work is supported by a public grant overseen by the French National Research Agency (ANR) as part of the ``Investissement d'Avenir'' program, through the iCODE project funded by the IDEX Paris-Saclay, ANR-11-IDEX-0003-02. The second author was also partially supported by ANR PIA funding number ANR-20-IDEES-0002.

\bibliographystyle{abbrv}
\bibliography{Bib}

\end{document}